\documentclass[11pt]{article}

\usepackage{amsmath}
\usepackage[font=small]{caption}
\usepackage{amsthm}
\usepackage{amsfonts}
\usepackage[pdftex]{graphicx}
\usepackage[caption=false]{subfig}
\usepackage[alphabetic]{amsrefs}
\usepackage{amssymb}                                                                                                                                                                                                                                                                                                                                                                                                                                                                                                                                                                                                                                                                                                                                                                                                 
\usepackage{pinlabel}
\usepackage{morefloats}
\usepackage{enumerate}
\usepackage{fullpage}
\usepackage{array}

\theoremstyle{plain} \newtheorem{thm}{Theorem}
\theoremstyle{plain} \newtheorem{cor}[thm]{Corollary}
\theoremstyle{plain} 
\theoremstyle{plain} \newtheorem{conj}[thm]{Conjecture}
\theoremstyle{plain} \newtheorem{lem}[thm]{Lemma}
\theoremstyle{definition} 
\theoremstyle{remark} \newtheorem{rmk}[thm]{Remark}
\theoremstyle{remark}

\newcommand{\BSm}{\Sigma \left( p, q, pqn - 1 \right)}
\newcommand{\BSp}{\Sigma \left( p, q, pqn + 1 \right)}

\newcommand{\Tpq}{T_{ p,q } }

\newcommand{\surg}[2]{S^{3}_{{#1}}\left({#2}\right)}

\newcommand{\s}{\mathfrak{s}}
\newcommand{\Tp}[1]{\mathcal{T}^{+}_{#1}}
\newcommand{\dee}[2]{d( {#1}, {#2})}

\newcommand{\ZU}{\mathbb{Z}[U]}
\newcommand{\ZUU}{\mathbb{Z}[U, U^{-1}]}
\newcommand{\Q}{\mathbb{Q}}

\newcommand{\HFps}[2]{HF^{+}\left( {#1}, {#2} \right)}
\newcommand{\HFp}[1]{HF^{+}\left({#1} \right)}
\newcommand{\HFrs}[2]{HF_{red}^{+} \left( {#1}, {#2} \right)}
\newcommand{\HFr}[1]{HF_{red}^{+} \left( {#1} \right)}
\newcommand{\HFis}[2]{HF^{\infty} \left( {#1}, {#2} \right)}
\newcommand{\HFm}[1]{HF^{-}\left({#1} \right)}
\newcommand{\HFi}[1]{HF^{\infty} \left( {#1}\right)}
\newcommand{\HFe}[1]{HF_{even}^{+} \left( {#1}\right)}
\newcommand{\HFo}[1]{HF_{odd}^{+} \left( {#1}\right)}
\newcommand{\OS}{Ozsv\'ath and Szab\'o }
\newcommand{\nem}{N\'emethi }

\newcommand{\Spq}{\mathcal{S}_{p,q}}
\newcommand{\Zp}{\mathbb{Z}_{\geq 0}}
\newcommand{\Ztwo}{\mathbb{Z}/2\mathbb{Z}}

\newcommand{\ep}{\varepsilon}

\newcommand\T{\rule{0pt}{2.6ex}}
\newcommand\B{\rule[-1.2ex]{0pt}{0pt}}

\begin{document}

\title{Heegaard Floer homology and several families of Brieskorn spheres}
\author{Eamonn Tweedy\\
Rice University\\
\texttt{eamonn@rice.edu}}

\maketitle

\begin{abstract}
In \cite{os:plumb}, \OS gave a combinatorial description for the Heegaard Floer homology of boundaries of certain negative-definite plumbings. \nem constructed a remarkable algorithm in \cite{nem:plumb} for executing these computations for almost-rational plumbings, and his work in \cite{nem:minus} gives a formula computing the invariants for the Brieskorn homology spheres $-\BSp$.  Here we give a formula for $\HFp{-\BSm}$, generalizing the one for the $n=1$ case given in \cite{nem:plus}.  We also compute $HF^{+}$ for the families $-\Sigma(2,5,k)$ and $-\Sigma(2,7,k)$.
\end{abstract}
\section{Introduction}

The Heegaard Floer homology package was first introduced by \OS \cite{os:disk}, and has proven to be a useful collection of tools for the study of manifolds of dimensions 3 and 4.  In particular, to a closed, oriented 3-manifold $Y$ one can associate a graded $\ZU$-module $\HFp{Y}$, which is the richest of the flavors (i.e. carries the most data).  Combinatorial techniques and cut-and-paste techniques have since been developed to ease computation of various Heegaard Floer homologies (the methods in \cite{lot:BFH}, \cite{j:SFH}, and \cite{sw:CHF} compute the graded $\mathbb{Z}$-module $\widehat{HF}$ and those in \cite{mot} compute $HF^{+}$).

However, \OS offered in \cite{os:plumb} a combinatorial description of  $HF^{+}$ for a 3-manifold which bounds a negative-definite plumbing graph. \nem provided a nice combinatorial framework for this description in \cite{nem:plumb} via a very concrete algorithm:  a plumbing graph leads to a function $\tau: \Zp \rightarrow \mathbb{Z}$, which in turn induces an intermediate object known as a ``graded root''; this gadget both determines $HF^{+}$ and carries some extra data relevant to singularity theory.  For a very approachable ``user's guide'' to N\'emethi's algorithm, see \cite{kar:maz}.

The purpose of the present article is to write down formulae for the invariants $HF^{+}$ associated to some Brieskorn homology spheres; these include $-\BSm$, where $p, q, n \in \mathbb{N}$ with $p,q$ coprime (Theorem \ref{thm:minus}) and $-\Sigma(2,j,k)$ where $j,k \in \mathbb{N}$, $k$ is odd and coprime and $j \in \{ 5,7 \}$ (Theorem \ref{thm:57}).  These computations attempt to further illustrate the usefulness of N\'emethi's formula in allowing one to compute $HF^{+}$ for such infinite families of Seifert manifolds.

Recall that 
$$\surg{1/n}{\Tpq} = -\BSm \quad \text{and} \quad \surg{-1/n}{\Tpq} = \BSp,$$
where $\Tpq$ denotes the right-handed $(p,q)$ torus knot and $\surg{r}{K}$ denotes $r$-framed surgery on the knot $K \subset S^{3}$.

Negative Dehn surgeries on algebraic knots were extensively studied in \cite{nem:minus}, and \nem writes down a closed formula for $\HFp{-\surg{-1/n}{K}}$ in $\S$ 5.6.2 of that paper.  The torus knot $\Tpq$ is indeed algebraic, and N\'emethi's formula gives that when $n>0$,
\begin{gather}
\begin{aligned}
\HFe{-\surg{-1/n}{\Tpq}} &= \Tp{0}(\alpha_{g - 1})^{\oplus n} \oplus \displaystyle \bigoplus_{i = 1}^{n(g - 1)}\label{eqn:plus}
\Tp{\left( \lfloor i/n \rfloor + 1 \right) \left( \{ i/n \}n + i \right)} \left( \alpha_{g - 1 + \lceil i/n \rceil} \right)^{\oplus 2},\\
\HFo{-\surg{-1/n}{\Tpq}} &= 0, \quad \text{and} \quad d \left( -\surg{-1/n}{\Tpq} \right) = 0
\end{aligned}
\end{gather}
Note that $\left\{ x \right\}:=x - \lfloor x \rfloor$;  definitions of other objects and notations involved can be found $\S$\ref{sec:pre}.

The case of $+1$-surgery on $\Tpq$ was studied by Borodzik and \nem in \cite{nem:plus}, and a formula for $HF^{+}$ was given there.  Presently, we extend that computation to $+1/n$-surgery for $n > 0$ to provide the following formula.

\begin{thm}\label{thm:minus}
Let $p,q > 0$ be coprime integers, and let $\Tpq$ denote the torus right-handed $(p,q)$ torus knot.  Then
\begin{align*}
\HFe{\surg{1/n}{\Tpq}} &= \Tp{-2\alpha_{g-1}}(\alpha_{g - 1})^{\oplus (n-1)} \oplus \displaystyle \bigoplus_{i = 1}^{n(g - 1)} 
\Tp{\left( \lceil i/n \rceil \right) \left( \{ (i-1)/n \}n + i - 1 \right) - 2 \alpha_{g - 1 + \lceil i/n \rceil}} \left( \alpha_{g - 1 + \lceil i/n \rceil} \right)^{\oplus 2}\\
\HFo{\surg{1/n}{\Tpq}} &= 0, \quad \text{and} \quad d \left( \surg{1/n}{\Tpq} \right) = -2 \alpha_{g - 1}
\end{align*}
\end{thm}

\nem's method doesn't rely on the surgery presentation given above, and may be applied to many other infinite families of Seifert manifolds.  For the sake of illustration, we compute $HF^{+}$ for all Brieskorn homology spheres of the form $-\Sigma(2,5,k)$ or $-\Sigma(2,7,k)$.  This computation is provided by the following (along with Equation \ref{eqn:plus} and Theorem \ref{thm:minus}), which is proven in $\S$\ref{sec:57}.  The main technical input comes from Lemma \ref{lem:57}, which is stated and proved in $\S$\ref{app}.  

\begin{thm}\label{thm:57}
Fix $n \in \mathbb{N}$, and let $M$ be any of the Brieskorn homology spheres $-\Sigma(2,5,10n\pm3)$, $-\Sigma(2,7,14n\pm3)$, or $-\Sigma(2,7,14n\pm5)$.  Then $\HFp{M}$ is as characterized by Table \ref{tab:HF}.
\end{thm}

\begin{table}[h]
\noindent\makebox[\textwidth]{
	\begin{tabular}[1.5\textwidth]{|c|c|c|c|}
	\hline
	manifold  & \T\B$HF^{+}_{red} \cong HF^{+}_{even}$ & $d$ \\ \hline
	$-\Sigma(2,5,10n - 1)$ & $ \Tp{-2}(1)^{\oplus(n-1)} \oplus \left( \displaystyle\bigoplus_{i = 0}^{n-1} \Tp{2i}(1)^{\oplus2} \right)$  & $-2$\\ \hline
	$-\Sigma(2,5,10n + 1)$ & $\Tp{0}(1)^{\oplus n} \oplus \left( \displaystyle\bigoplus_{i = 1}^{n} \Tp{2i}(1)^{\oplus2} \right)$ &$0$ \\ \hline
	$-\Sigma(2,5,10n-3)$ & $\displaystyle \Tp{0}(1)^{\oplus(n-1)} \oplus \left( \displaystyle\bigoplus_{i = 0}^{n-1} \Tp{2i}(1)^{\oplus2} \right)$ & $0$\\ \hline
	$-\Sigma(2,5,10n+3)$ & \T\B$\displaystyle\Tp{-2}(1)^{\oplus n} \oplus \left( \displaystyle\bigoplus_{i = 0}^{n-1} \Tp{2i}(1)^{\oplus2} \right)$  & $-2$ \\ \hline
	$-\Sigma(2,7,14n-1)$ & $\Tp{-4}(2)^{\oplus(n-1)} \oplus \left( \displaystyle\bigoplus_{i = 0}^{n-1} \Tp{2i}(1)^{\oplus2} \right)\oplus \left( \displaystyle\bigoplus_{i = 0}^{n-1} \Tp{2n + 4i}(1)^{\oplus2} \right)$ & $-4$\\ \hline
	$-\Sigma(2,7,14n+1)$ & $\Tp{0}(2)^{\oplus(n)} \oplus \left( \displaystyle\bigoplus_{i = 1}^{n} \Tp{2i}(1)^{\oplus2} \right)\oplus \left( \displaystyle\bigoplus_{i = 1}^{n} \Tp{2n + 4i}(1)^{\oplus2} \right)$ & $0$  \\ \hline
	$-\Sigma(2,7,14n-3)$ & \T\B $\Tp{-2}(1)^{\oplus(2n-2)} \oplus \left( \displaystyle\bigoplus_{i = 0}^{n-1} \Tp{2i-2}(1)^{\oplus2} \right)\oplus \left( \displaystyle\bigoplus_{i = 0}^{n-1} \Tp{2n + 4i - 2}(1)^{\oplus2} \right)$ & $-2$ \\ \hline
	$-\Sigma(2,7,14n+3)$ & \T\B $\Tp{0}(1)^{\oplus(2n+1)} \oplus \left( \displaystyle\bigoplus_{i = 1}^{n} \Tp{2i}(1)^{\oplus2} \right)\oplus \left( \displaystyle\bigoplus_{i = 1}^{n} \Tp		{2n + 4i}(1)^{\oplus2} \right)$& $0$ \\ \hline
	$-\Sigma(2,7,14n-5)$ & \T\B $\Tp{-2}(1)^{\oplus(2n-3)} \oplus \left( \displaystyle\bigoplus_{i = 0}^{n-1} \Tp{2i-2}(1)^{\oplus2} \right)\oplus \left( \displaystyle\bigoplus_{i = 0}^			{n-1} \Tp{2n +4i - 2}(1)^{\oplus2} \right)$ & $-2$ \\ \hline
	$-\Sigma(2,7,14n+5)$ & \T\B $\Tp{0}(1)^{\oplus(2n+2)} \oplus \left( \displaystyle\bigoplus_{i = 1}^{n} \Tp{2i}(1)^{\oplus2} \right)\oplus \left( \displaystyle\bigoplus_{i = 1}^{n} \Tp		{2n + 4i}(1)^{\oplus2} \right)$ & $0$ \\ \hline
 	\end{tabular}}
	\caption{The structure of $HF^{+}$ for lots of Brieskorn spheres, as provided by Theorems \ref{thm:minus} and \ref{thm:57} and Bordzik and N\'emethi's Equation \ref{eqn:plus}.}
	\label{tab:HF}
\end{table}

Let $p$ be a prime and $K \subset S^{3}$ be a knot.  Using a particular $d$-invariant for the $p^{n}$-fold cover of $S^{3}$ branched along $K$, Manolescu and Owens \cite{cmo:delta} (for $p^{n} = 2$) and Jabuka \cite{jab:delta} (for any prime $p$ and any $n \in \mathbb{Z}$) define a concordance invariant $\delta_{p^{n}} \mathbb{Z}$; see $\S$\ref{sec:delta} for the definition of this invariant.  Theorem \ref{thm:57} provides some new $\delta_{p^{n}}$-invariants for torus knots (although a few of the examples below appeared in \cite{jab:delta}).  Recall from \cite{mil:BS} that when $p,q,r > 0$ are pairwise coprime, in fact
$$ -\Sigma(p,q,r) = \Sigma_{p}(T_{q,r}) = \Sigma_{q}(T_{p,r}) = \Sigma_{r}(T_{p,q}).$$
\begin{cor}\label{cor:delta}
For $p,q \in \mathbb{N}$ coprime, let $\Tpq$ denote the right-handed $(p,q)$-torus knot.
\begin{enumerate}[(i)]
\item Let $k = 10n\pm3$, where $n \in \mathbb{N}$.  Then
$$ \delta_{2}\left( T_{5,k} \right) = \delta_{5}\left( T_{2,k} \right) = \begin{cases}
4, & k = 10n+3\\
0, & k = 10n-3
\end{cases}
$$
Moreover, whenever $k$ is a prime power,
$$ \delta_{k}\left( T_{2,5} \right) = \begin{cases}
-4, & k = 10n+3 \text{ (e.g. $n = 1,2,4,5,7, \ldots$})\\
0, & k = 10n-3\text{ (e.g. $n = 1,2,4,5,7, \ldots$})
\end{cases}
$$
\item Let $k = 14n\pm3$ or $k = 14n\pm5$, where $n \in \mathbb{N}$.  Then
$$ \delta_{2}\left( T_{7,k} \right) = \delta_{7}\left( T_{2,k} \right) = \begin{cases}
-4, & k = 14n-3 \text{ or } k = 14n-5\\
0, & k = 14n+3 \text{ or } k = 14n+5
\end{cases}
$$
Moreover, whenever $k$ is a prime power,
$$ \delta_{k}\left( T_{2,7} \right) = \begin{cases}
-4, & k = 14n-3 \text{ (e.g. $n = 1,4,5,8,10, \ldots$)}\\
0, & k = 14n+3\text{ (e.g. $n = 1,2,4,5,7, \ldots$)}\\
-4, & k = 14n+5\text{ (e.g. $n = 1,3,4,6,7, \ldots$)}\\
0,  &k=14n-5\text{ (e.g. $n = 2,3,6,8,11, \ldots$)}
\end{cases}
$$
\end{enumerate}
\end{cor}

\begin{rmk}
The families $T_{5,k}$ and $T_{7,k}$ provide several new infinite families of (non-alternating, of course) knots such that $-\delta_{2} \neq \sigma/2$ (c.f. \cite{cmo:delta}).  For the reader's convenience, we list those values here:
\begin{table}[h]
\centering
\begin{tabular}{|c|c|c|c|c|c|c|}
\hline
\T\B knot $K$ & $T_{5,10n+3}$ & $T_{5,10n-3}$ & $T_{7,14n+3}$&  $T_{7,14n-3}$& $T_{7,14n+5}$ & $T_{7,14n-5}$\\ \hline
$\sigma(K)/2 \T\B$ & $4(3n+1)$ & $4(3n-1)$ & $4(6n+1)$ & $4(6n-1)$ & $4(6n+2)$ & $4(6n-2)$ \\ \hline
$-\delta_{2}(K) \T \B$ & $4$ & $0$ & $0$ & $4$ & $0$ & $4$\\ \hline
\end{tabular}
\end{table}
\end{rmk}

\begin{conj}\label{conj:d}
Let $p>1$ be an odd integer and let $k$ be an integer with $\text{gcd}(k,2p) = 1$ and $k \not\equiv \pm1 \pmod{2p}$.
$,$ then
\begin{gather*}
\begin{aligned}
d\left( -\Sigma\left( 2,p,2pn - k \right) \right) &= \begin{cases}
	0, & p \equiv 1 \pmod{4}\\
	-2, & p \equiv 3 \pmod{4}
\end{cases} \quad \text{and}\\
d\left( -\Sigma\left( 2,p,2pn + k \right) \right) &= \begin{cases}
	0, & p \equiv 3 \pmod{4}\\
	-2, & p \equiv 1 \pmod{4}
\end{cases}
\end{aligned}
\end{gather*}
\end{conj}

\begin{rmk}
Along with Theorem \ref{thm:minus} and Equation \ref{eqn:plus}, Conjecture \ref{conj:d} would determine the invariant $\delta_{2}$ for all torus knots $\Tpq$ with $p,q$ odd.
\end{rmk}

\section*{Acknowledgements}
I would like to thank \c{C}a\u{g}ri Karakurt for introducing me to N\'emethi's graded roots during his visit to Rice in Spring 2012 (and he and his coauthor for the useful primer in \cite{kar:maz}).  I would also like to thank Tye Lidman for some helpful conversations.

\section{Preliminaries}\label{sec:pre}

\subsection{Seifert fibered integer homology spheres}\label{sec:SF}

Recall that the Seifert fibered space $\Sigma : = \Sigma\left(e_0, (a_1,b_1), (a_2,b_2), \ldots, (a_m,b_m)\right)$ bounds a plumbed negative-definite 4-manifold whose plumbing graph is star-shaped, consisting of $m$ ``arms'' emanating from a central vertex.  The central vertex is labelled with weight $e_0$, and the $i^{th}$ arm is a chain of $n_i$ vertices with labels $-k_1, -k_2, \ldots, -k_{n_i}$ (ordered outward from the center), where

$$ \frac{a_i}{b_i} = k_{1} - \cfrac{1}{k_{2} - \cfrac{1}{\ddots - \cfrac{1}{k_{n_i}}}}$$

Now define the quantities
$$ e :=e_0 + \sum_{i = 1}^{m} \frac{b_{i}}{a_{i}} \quad \text{and} \quad \ep:=\frac{1}{e}\left(2-m + \sum_{i=1}^{m} \frac{1}{a_{i}} \right)$$
If we assume $e < 0$, then $\Sigma$ is an integer homology sphere if and only if $ e = -1/\left(a_{1}a_{2}\ldots a_{m}\right)$, i.e.
\begin{equation*}\label{eqn:SF}
-1 = e_{0}a_{1}a_{2}\ldots a_{m} + \sum_{i=1}^{m}b_{i}\left(\frac{a_{1}a_{2}\ldots a_{m} }{a_{i}}\right)
\end{equation*}
This equation implies that the residue of $b_{i}$ modulo $a_{i}$ is determined by the $a_{j}$'s.

\subsection{Torus knots}

We review some notation related to the torus knot $\Tpq$ which appears in some formulae in the introduction.  Let $\Spq \subset \Zp$ denote the semigroup
$$\Spq := \left\{ ap + bq \ \big| \ (a,b) \in \Zp^{2}  \right\}.$$

Now $\Zp \setminus \Spq$ is finite, and in fact
$$\big| \Zp \setminus \Spq \big| = \frac{(p-1)(q-1)}{2} = g_{3}(\Tpq) =: g, \quad \text{the 3-genus of $\Tpq$}.$$
For $i \geq 0$, we define a sequence of numbers $\alpha_{i} \in \Zp$ via
$$\alpha_{i} :=  \# \left\{ s \notin \Spq \ \big| \ s > i \right\}.$$
One can verify that in fact $g = \alpha_{0} \geq \alpha_{1} \geq \ldots \geq \alpha_{2g - 3} \geq \alpha_{2g-2} = 1$ and $\alpha_{i} = 0$ for $i > 2g - 2$.

\subsection{Dedekind sums}

Recall the ``sawtooth function'' $\langle \cdot \rangle: \mathbb{R} \rightarrow \mathbb{R}$, where
$$ \langle x \rangle := \begin{cases}
	0, & x \in \mathbb{Z}\\
	x - \lfloor x \rfloor - \frac{1}{2}, & x \notin \mathbb{Z}
\end{cases}$$
For $h,k \in \mathbb{Z}\setminus\{0\}$, one can define the classical \textbf{Dedekind sum}
$$ s(h,k):= \sum_{i = 1}^{k-1} \left\langle \frac{i}{k} \right\rangle \left\langle \frac{hi}{k} \right\rangle$$

We'll make use of a particular formula found in \cite{apo:ded} involving the Euclidean algorithm.  Assume that $0<h<k$ and that $r_{0}, r_{1}, \ldots, r_{n+1}$ are the remainders obtained when the Euclidean algorithm is applied to $h$ and $k$, i.e.
$$ r_{0} := k, \quad r_{1}: = h, \quad r_{j+1} \equiv r_{j-1} \pmod {r_{j}} \quad (\text{with }1 \leq r_{j+1} < r_{j}),\quad \text{and} \quad r_{n+1} = 1.$$
Then in fact the Dedekind sum can be computed via
\begin{equation}\label{eqn:DS}
s(h,k) = \frac{1}{12} \left( \sum_{j = 1}^{n+1} (-1)^{j+1}\left(\frac{1 + r_{j}^{2} + r_{j-1}^{2}}{r_{j} r_{j-1}} \right)\right) - \frac{1 + (-1)^{n}}{8}
\end{equation}

\subsection{Heegaard Floer homology}\label{sec:HF}

Let $Y$ be a rational homology 3-sphere, and fix $\mathfrak{s} \in Spin^{c}(Y)$.  We study the $\mathbb{Q}$-graded Heegaard Floer homology groups $\HFps{Y}{\s}$, defined by \OS in \cite{os:disk}.  Define the graded $\ZU$-modules
$$ \Tp{} := \frac{\ZUU}{U \cdot \ZU} \quad \text{and} \quad \Tp{}(n) : = \frac{\mathbb{Z}\langle U^{-n+1}, U^{-n+2}, \ldots \rangle}{U \cdot \ZU}, \quad \text{where} \quad \text{deg}\left( U^{k} \right) = -2k.$$

More generally, given a graded $\ZU$-module $M$ with $k$-homogeneous elements $M_{k}$ and some $d \in \Q$, let $M[d]$ denote the graded $\ZU$ module with $M[d]_{(k + d)} = M_{k}$.  Then define the shifted modules $ \Tp{d} := \Tp{}[d]$, $\Tp{d}(n): = \Tp{}(n)[d]$.  Recall that the $\Q$-graded Heegaard Floer groups decompose as
$$ \HFps{Y}{s} \cong \Tp{\dee{Y}{\s}} \oplus \HFrs{Y}{\s},$$
where the first summand is the image of the projection map $\HFis{Y}{\s} \rightarrow \HFps{Y}{\s}$ and second is its quotient.  Note that the invariant $\dee{Y}{\s} \in \Q$ is the so-called \textbf{correction term} or \textbf{d-invariant} associated to the pair $\left( Y, \s \right)$, first introduced in \cite{os:abs}.  Note that when $Y$ is an integer homology sphere, there is only one element in $Spin^{c}(Y)$; in this case, suppress the ``$\mathfrak{s}$'' and just write $d(Y)$.

Recall also that $HF^{+}$ is relatively $\mathbb{Z}$-graded and carries a well-defined absolute $\Ztwo$-grading.  With respect to this grading, $HF^{+}_{red}$ further decomposes as
$$ \HFr{Y} = \HFo{Y}\oplus\HFe{Y}.$$

\begin{rmk}
Given $\HFp{-Y}$, it is straightforward to compute $\HFp{Y}$.  Indeed, one should first use the long exact sequence
$$ \ldots \longrightarrow \HFm{-Y} \longrightarrow \HFi{-Y} \longrightarrow \HFp{-Y} \longrightarrow \ldots$$
to recover $\HFm{-Y}$, and then use the fact that $ HF_{*}^{+}(Y) \cong HF_{-}^{(-*-2)}(-Y).$
\end{rmk}

\subsection{N\'emethi's algorithm}

In \cite{nem:plumb}, N\'emethi describes a procedure for computing the Heegaard Floer homology for boundaries of negative-definite \textbf{almost-rational} plumbings.  A plumbing is almost-rational if its graph contains \emph{at most one} \textbf{bad vertex}, i.e. a vertex $v$ with $\text{degree}(v) > | \text{weight}(v) |$.  Note that the star-shaped plumbing graph associated to a Seifert manifold can always be drawn such that no vertices are bad except possibly the central one.  We'll briefly describe the algorithm; see \cite{kar:maz} for a very concrete user's guide.

Let $\Gamma$ be the plumbing graph and let $X(\Gamma)$ denote the associated plumbed $4$-manifold.  The algorithm uses $\Gamma$ to induce a \textbf{computation sequence} of vectors in $H_{2}(X(\Gamma))$, which in turn provides a \textbf{tau function} $\tau: \Zp \rightarrow \mathbb{Z}$.  One then constructs a reduced version $\widetilde{\tau}$ of the function by throwing out all repetition in the sequence $\left( \tau(i)  \right)$, keeping only the local extrema, and re-indexing the result.  The function $\widetilde{\tau}$ generates a $\mathbb{Z}$-graded infinite tree called a \textbf{graded root}; this tree has an obvious $\ZU$-action and recovers the module $\HFp{-\partial X(\Gamma)}.$

\begin{rmk}
Given a particular plumbing graph, one would use the reduced function $\widetilde{\tau}$ to draw the graded root in practice; however, notice that Equation \ref{eqn:SFtau1} indeed involves the full tau function $\tau$, and that's the one commonly used in the arguments here.
\end{rmk}

\subsection{Some formulas of Bordzik and \nem}\label{sec:BN}

Let $\Sigma$ be a Seifert-fibered homology sphere (as $\S$\ref{sec:SF}).  In \cite{nem:plus},  Borodzik and \nem characterize the $d$-invariant and the tau function for $\Sigma$ in terms of its Seifert invariants $(e_{0}, (a_{1},b_{1}), \ldots, (a_{m},b_{m}))$.  Proposition 2.2 of \cite{nem:plus} states that for each $k \in \Zp$,
\begin{equation}\label{eqn:SFtau1}
\tau(k) = \sum_{j=0}^{k-1} \triangle_{j}, \quad \text{where} \quad \triangle_{j} := 1 - je_{0} - \sum_{i = 1}^{m}\left\lceil \frac{j b_{i}}{a_{i}} \right\rceil
\end{equation}
There is an alternate form for $\triangle_{j}$ which is sometimes more useful.  For any $b \in \mathbb{Z}$ and $a_{1}, \ldots, a_{m} \in \mathbb{Z}_{>0}$, let
$$ \ep_{a}(b) := \sum_{i = 1}^{m} \ep_{a_{i}}(b), \quad \text{where} \quad
 \ep_{a_{i}}(b) := \begin{cases}
	1, & a_{i} | b\\
	0, & \text{else}
\end{cases}$$
Then we have that
\begin{equation}\label{eqn:SFtau2}
\triangle_{j} = 1 - \frac{m}{2} + \frac{j}{a_{1} \ldots a_{m}} + \frac{\ep_{a}(j)}{2} + \sum_{i = 1}^{m} \left\langle \frac{jb_{i}}{a_{i}} \right\rangle
\end{equation}
The $d$-invariant is then given by
\begin{equation}\label{eqn:SFd}
d(\Sigma) = \frac{1}{4} \left( \ep^2 e + e + 5 - 12 \sum_{i = 1}^{m} s(b_{i}, a_{i}) \right) - 2 \min_{k\geq 0} \tau(k)
\end{equation}

\subsection{The concordance invariant $\delta_{p^{n}}$}\label{sec:delta}

Let $K \subset S^{3}$ be a knot, let $p$ be a prime, and let $n \in \mathbb{N}$.  Then let $\Sigma_{p^{n}}(K)$ denote the $p^{n}$-fold branched cover of $S^{3}$ branched along the knot $K$.  $\Sigma_{p^{n}}(K)$ is a rational homology sphere, and we let $\mathfrak{s}_{0} \in Spin^{c}(\Sigma_{p^{n}}(K))$ denote the element induced by the unique $spin$-structure.  Manolescu and Owens \cite{cmo:delta} (for $p^{n} = 2$) and Jabuka \cite{jab:delta} (for general $p^{n}$) define
$$\delta_{p^{n}}(K) := 2d\left(\Sigma_{p^{n}}(K), \mathfrak{s}_{0} \right) \in \mathbb{Z}.$$
This number is an invariant of the smooth knot concordance class of $K$, and in fact provides a homomorphism from the smooth knot concordance group to $\mathbb{Z}$.  Corollary \ref{cor:delta} provides some new computations of $\delta_{p^{n}}$ for some torus knots.

\section{The Brieskorn spheres $\BSm$}\label{sec:BSm}

Let $p,q > 0$ be coprime integers.  Recall that the Brieskorn homology sphere $\Sigma(p,q,pqn-1)$ is a Seifert fibered space with Seifert invariants $(e_{0}, (p, p'), (q, q'), (r,r'))$, where $e_0 = -2$, $r = pqn - 1$, $r' = pqn - n -1$, and $p',q'$ uniquely determined by the restrictions
\begin{equation*}
0 < p' < p, \quad 0 < q' < q, \quad pq' \equiv 1 \pmod q, \quad \text{and} \quad qp' \equiv 1 \pmod p.
\end{equation*}

\begin{rmk}
When $p  = 2$, the above constraints imply that $p ' = 1$ and $q' = (q+1)/2$.  The reader can verify that these parameters lead to the plumbing graph found in Figure \ref{fig:m} in $\S$\ref{app}.  Note that we don't need the graphs to compute the Heegaard Floer groups, but rather only the Seifert invariants.
\end{rmk}

\nem gives a formula for the function $\tau$ for a Seifert manifold in \cite{nem:plumb}, and in \cite{nem:plus} uses it to compute $HF^{+}$ for $+1$-surgery on $\Tpq$.  In order to extend his result result to $+1/n$-surgery ($n \in \mathbb{N})$, we first characterize the function $\tau$ for this case.

\begin{lem}\label{lem:tau1}
\begin{enumerate}[(i)]
Let $\Spq$ denote the semigroup of $\Zp$ generated by $p$ and $q$.  Additionally, let $N:=n(2g - 1)$.  Then the following hold.
\item The function $\tau: \Zp \rightarrow \mathbb{Z}$ attains its local maxima (resp. minima) at the points $M_{i}$ (resp. $m_{i}$), where
$$ M_{i} := pqi + 1 \text{ (for $0 \leq i \leq N-2$)} \quad \text{and} \quad m_{i} := pqi - \lfloor i / n \rfloor \text{  (for $0 \leq i \leq N-1$)}.$$
\item For $0 \leq i \leq N - 2$,
\begin{align}
\label{eqn:tau1}
\tau( M_{i}) - \tau(m_{i}) &=\# \left\{ s \in \Spq \ \Big| \ s \leq \left\lfloor \frac{i}{n} \right\rfloor \right\} > 0  \quad \text{and}\\
\label{eqn:tau2}
\tau( M_{i}) - \tau(m_{i+1}) &= \# \left\{ s \notin \Spq \ \Big| \ s \geq \left\lfloor \frac{i+1}{n} \right\rfloor + 1 \right\} > 0.
\end{align}
\item The sequence $\left( m_i\right)$ satisfies
\begin{equation*}\label{eqn:tau3}
\tau\left( m_{i+1} \right) - \tau \left( m_i \right)  \begin{cases} \leq 0 & \text{for }  i \in \left\{ 0, \ldots, \frac{N-n}{2}-1 \right\}\\
= 0 & \text{for } i \in \left\{ \frac{N-n}{2}, \ldots, \frac{N +n}{2} - 2 \right\}\\
\geq 0 & \text{for } i \in \left\{ \frac{N+n}{2} - 1, \ldots, N - 2 \right\}\end{cases}
\end{equation*}
and thus $\tau$ achieves its global minimum value at the points $m_{i}$ with $i \in \left\{ \frac{N-n}{2}, \ldots, \frac{N+n}{2}-1 \right\}.$
\end{enumerate}
\end{lem}

\begin{proof}[Proof of (i)]
\renewcommand{\qedsymbol}{}
For each $i \in \Zp$, define the numbers $M_{i}$ and $m_{i}$ via the expressions stated in the lemma (we'll show that $\tau$ attains its local extrema at some of these points, i.e. the ones with indices restricted as in the lemma).  To this end, fix $j \in \Zp$ and compute $\triangle_{j} : = \tau(j+1) - \tau(j)$.  We'll first assume that $M_{i} \leq j < M_{i+1}$, and employ an analysis similar to that in \cite{nem:plus}.  In particular, recall that for any integer $s \in [0,pq)$,
\begin{align*}
	s \in \Spq &\iff s = \alpha p + \beta q \quad \text{for some} \quad 0\leq \alpha < q, 0 \leq \beta < p \quad \text{and}\notag\\
	s \notin \Spq &\iff s + pq =  \alpha p + \beta q \quad \text{for some} \quad 0\leq \alpha < q, 0 \leq \beta < p.\label{eqn:Spq}
\end{align*}
Following \cite{nem:plus}, one can use this fact along with Equation \ref{eqn:SFtau2} to show that
$$ \triangle_j = \begin{cases}
\left\lfloor \frac{jn}{pqn-1}  \right\rfloor - i & \text{when} \quad (i+1)pq - j \in \Spq\\[.5em]
\left\lfloor \frac{jn}{pqn-1}  \right\rfloor - i - 1 & \text{when} \quad (i+1)pq - j \notin \Spq
\end{cases}.$$
First assume that $0 \leq i \leq N-2$.  In this case, whenever $ (i+1)pq - j \in \Spq$, one finds that
\begin{equation}\label{eqn:delta1}
\left\lfloor \frac{jn}{pqn-1}  \right\rfloor - i = \begin{cases} 0 & \text{when} \quad M_{i} \leq j < m_{i+1}\\ 1 & \text{when} \quad m_{i+1} \leq j < M_{i+1}  \end{cases}.
\end{equation}
On the other hand, whenever $(i+1)pq - j \notin \Spq$,
\begin{equation}\label{eqn:delta2}
\left\lfloor \frac{jn}{pqn-1}  \right\rfloor - i - 1 = \begin{cases} -1 & \text{when} \quad M_{i} \leq j < m_{i+1}\\ 0 & \text{when} \quad m_{i+1} \leq j < M_{i+1}  \end{cases}.
\end{equation}
Now when $i \geq N-1$, we find that $\triangle_j \geq 0$ regardless of whether $(i+1)pq - j \in \Spq$.
\end{proof}
\begin{proof}[Proof of (ii)]
\renewcommand{\qedsymbol}{}
Equations \ref{eqn:tau1} and \ref{eqn:tau2} follow from equations \ref{eqn:delta1} and \ref{eqn:delta2}, bearing in mind that
\begin{align*}
m_{i} \leq j \leq M_{i} \quad &\iff \quad pq-1 \leq (i+1)pq -  j \leq pq + \left\lfloor \frac{i}{n} \right\rfloor \quad \text{and}\\
M_{i}  \leq j \leq m_{i+1} \quad &\iff \quad \left\lfloor \frac{i + 1}{n} \right\rfloor + 1 \leq (i + 1)pq - j \leq pq - 1
\end{align*}
Both quantities are strictly positive because $2g-1 \notin \Spq$ and $0 \in \Spq$.
\end{proof}
\begin{proof}[Proof of (iii)]
Observe that $k \in \Spq \iff 2g-1-k \notin \Spq$.  Along with equations \ref{eqn:tau1} and \ref{eqn:tau2}, this implies that
$$ \tau\left( m_{i+1} \right) - \tau \left( m_i \right) = \# \left\{ k \notin \Spq \ \bigg| \ k \geq 2g-1- \left\lfloor \frac{i}{n} \right\rfloor\right\}
-  \# \left\{ k \notin \Spq \ \bigg| \ k \geq \left\lfloor \frac{i+1}{n} \right\rfloor +1 \right\},$$
and the result follows.
\end{proof}
In fact, the graded root determined by the function $\tau$ is highly symmetric; the following makes this more precise.

\begin{lem}\label{lem:tau2}
Let $nk \leq i < nk + n$ for some $0 \leq k \leq g - 2$.
\begin{enumerate}[(i)]
\item ``Branch lengths'' are symmetric, i.e.
$$\tau\left( M_{n(g - 1) - 1 - i} \right) - \tau\left( m_{n(g - 1) - 1 - i} \right) = \tau\left( M_{ng + i - 1} \right) - \tau\left( m_{ng + i} \right) = \alpha_{g + k}$$

\item ``Leaf heights'' are symmetric, i.e.
$$2\tau\left( m_{n(g - 1) - 1 - i} \right) =  2\tau\left( m_{ng + i} \right) = \left( k + 1 \right) \left( 2i - nk \right) -2 \alpha_{g + k}+ C(n,g),$$
where $C(n,g) = g(n - ng + 2)$.
\end{enumerate}
Moreover, there is a ``bunch'' of $n$ leaves at the bottom level, i.e. for $n(g - 1) \leq i \leq ng - 2$,
$$\tau( M_{i}) - \tau(m_{i}) = \alpha_{g - 1} \quad \text{and} \quad 2\tau(m_{i}) = -2 \alpha_{g -1}+ C(n,g).$$
\end{lem}

\begin{proof}[Proof of (i)]
\renewcommand{\qedsymbol}{}
Follows from equation \ref{eqn:tau1}.
\end{proof}
\begin{proof}[Proof of (ii)]
We have that 
\begin{gather}
\begin{aligned}
\tau(M_0) = 1 \quad  \text{and}  \quad \tau(M_{i+1}) - \tau(M_{i}) = \left\lfloor \frac{i+1}{n} \right\rfloor + 1 - g \quad \text{for} \quad 0 \leq i \leq N-3,\\
\label{eqn:tau4}
\text{and so} \quad  \tau(M_{i})  = 1 + \displaystyle \sum_{  m = 1}^{i}  \left( \left\lfloor \frac{m}{n} \right\rfloor + 1 - g\right) \quad \text{for} \quad 1 \leq i \leq N-2.
\end{aligned}
\end{gather}
The statement follows from direct computations using equation \ref{eqn:tau4}.
\end{proof}

\begin{proof}[Proof of Theorem \ref{thm:minus}]
For $0 \leq i \leq n(g-1)$, let $G_{i}$ be given by 
$$G_{i}:= \left( \lceil i/n \rceil \right) \left( \{ (i-1)/n \}n + i - 1 \right) - 2\alpha_{g - 1 + \lceil i/n \rceil} + C(n,g).$$
Notice that lemmas \ref{lem:tau1} and \ref{lem:tau2} give us enough information about the function $\tau$ to conclude that for some constant shift $S$,
\begin{align*}
\HFr{\surg{1/n}{\Tpq}} &= \Tp{G_{0} + S}(\alpha_{g - 1})^{\oplus (n-1)} \oplus \displaystyle \bigoplus_{i = 1}^{n(g - 1)} 
\Tp{G_{i} + S} \left( \alpha_{g - 1 + \lceil i/n \rceil} \right)^{\oplus 2}\\
\text{and} \quad d \left( \surg{1/n}{\Tpq} \right) &= -2 \alpha_{g - 1} + C(n,g) + S.
\end{align*}
We claim that in fact $S + C(n,g) = 0$, which would finish the proof.  Recall that Moser showed in \cite{moser:lens} that $\surg{pq-1}{\Tpq}$ is a lens space.  In \cite{os:abs}, \OS computed $d$-invariants for surgeries on lens space knots (knots in $S^{3}$ for which there exist positive integer surgeries yielding lens spaces); in particular, for $n > 0$, the $d$-invariant of $1/n$ surgery on a lens space knot is independent of $n$.  Along with Equation \ref{eqn:plus}, this implies that
$$  d \left( \surg{1/n}{\Tpq} \right) = d \left( \surg{1}{\Tpq} \right) = -2\alpha_{g-1}.$$
\end{proof}
\begin{rmk}
One could alternately use Equations \ref{eqn:DS} and \ref{eqn:SFd} to compute the above $d$-invariant.
\end{rmk}

\section{The Brieskorn spheres $\Sigma(2,5,k)$ and $\Sigma(2,7,k)$}\label{sec:57}

We'll compute $HF^{+}$ for these manifolds using the formulae in $\S$\ref{sec:BN}, and the only inputs we'll need are the Seifert invariants (though the interested reader can see Figures \ref{fig:5m3}-\ref{fig:7m5} in $\S$\ref{app} for associated plumbing graphs).  The discussion in $\S$\ref{sec:SF} tells us that we can write
\begin{align*}
	\Sigma(2,5,10n-3) &= \Sigma\left(-1,(2,1),(5,1),(10n-3,3n-1)\right)\\
	\Sigma(2,5,10n+3) &= \Sigma\left(-2,(2,1),(5,4),(10n+3,7n+2)\right)\\
	\Sigma(2,7,14n-5) &= \Sigma\left(-2,(2,1),(7,5),(14n-5,11n-4)\right)\\
	\Sigma(2,7,14n-3) &= \Sigma\left(-2,(2,1),(7,6),(14n-3,9n-2)\right)\\
	\Sigma(2,7,14n+3) &= \Sigma\left(-1,(2,1),(7,1),(14n+3,5n+1)\right)\\
	\Sigma(2,7,14n+5) &= \Sigma\left(-1,(2,1),(7,2),(14n+5,3n+1)\right)
\end{align*}

The structure of $HF^{+}$ can be read off from the tau function.  We state and prove Lemma \ref{lem:57} in $\S$\ref{app}, which characterize $\tau$ for $\Sigma(2,5,k)$ and $\Sigma(2,7,k)$.  With those results in mind, we have the necessary ingredients for proving Theorem \ref{thm:57}.

\begin{proof}[Proof of Theorem \ref{thm:57}]
We prove the statement for $-\Sigma(2,7,14n+3)$, and leave the proofs for the other five families as exercises.

First notice in Table \ref{tab:tau5} that the local extrema of the tau function occur symmetrically, with
\begin{gather*}
\begin{aligned}
\tau\left(m_{6n+1-i}\right) &= \tau\left(m_i\right) = \begin{cases}
-2i, & i \in [0,n]\\
-n-i, & i \in [n+1, 2n]
\end{cases}
 \quad \text{and}\\
 \tau\left(M_{6n-i}\right) &= \tau\left( M_{i} \right) = \begin{cases}
 1 - 2i, & i \in [0,n]\\
 1-n-i & i \in [n+1,2n]
\end{cases}
 \end{aligned}
 \end{gather*}
 Lemma \ref{lem:57}, along with this observation, implies that for some shift $S \in \mathbb{Z}$,
\begin{align*}
		\HFr{-\Sigma(2,7,14n+3)} &= \Tp{-6n+S}(1)^{\oplus(2n-1)} \oplus \left( \displaystyle\bigoplus_{i = 1}^{n} \Tp{-6n+2i + S}(1)^{\oplus2} \right)\oplus \left( \displaystyle\bigoplus_{i = 1}^{n} \Tp{-4n+4i + S}(1)^{\oplus2} \right),\\
		 \text{and} \quad d\left(-\Sigma(2,7,14n+3) \right) &=-6n+S.
\end{align*}
Unfortunately, these manifolds aren't surgeries on torus knots.  Fortunately, it is straightforward to compute the $d$-invariants directly via Equation \ref{eqn:SFd}.  Applying the Euclidean algorithm to $k = 14n+3$ and $h = 5n+1$, one obtains the remainder sequences
\begin{equation*}
(14n+3, 5n+1, 4n+1,n, n-1,1) \text{ for } n>2, \quad (31,11,9,2,1)  \text{ for } n =2, \quad \text{and} \quad (17,6,5,1) \text{ for } n =1.
\end{equation*}
With these (and Maple) in hand, we find that
$$ d\left(-\Sigma(2,7,14n+3) \right) = -\left(\frac{1}{4}(-24n) - 2(-3n)\right) = 0 \quad \text{and so} \quad S = 6n.$$
\end{proof}

\newpage

\section{Appendix}\label{app}
The following provides the structure of the tau functions for the Brieskorn spheres in Theorem \ref{thm:57}.
\begin{lem}\label{lem:57}
\begin{enumerate}[(i)]  Fix $n \in \mathbb{N}$.  For the tau function of $\Sigma(2,5,10n\pm3)$ (resp. $\Sigma(2,7,14n\pm3)$, resp. $\Sigma(2,7,14n\pm5)$), the following hold.
	\item The function $\tau: \Zp \rightarrow \mathbb{Z}$ attains its local maxima  at the points $M_{i}$ and local minima at the points $m_{i}$, where these sequences are defined in Table \ref{tab:tau5} (resp \ref{tab:tau73}, resp. \ref{tab:tau75}).
	\item Changes in $\tau$ between consecutive extrema and the values of the function at these extrema are as given in Table \ref{tab:tau5} (resp \ref{tab:tau73}, resp. \ref{tab:tau75}).
\end{enumerate}
\end{lem}

\begin{table}[h]
	\centering
	\begin{tabular}{|m{1.2in}|c|c|}
	\hline
		\begin{center}manifold\end{center} & $\Sigma(2,5,10n-3)$ & $\Sigma(2,5,10n+3)$\\ \hline
		\begin{center} $M_{i}$\end{center} &
				$10i+1,  i \in [0,3n-2]$ & $\begin{cases}
					10i+1, &i \in [0,2n-1]\\
					10i+7,  &i \in [2n,3n-1]
				\end{cases}$ \\ \hline
		\begin{center}$m_{i}$\end{center} &
				$\begin{cases}
					0, & i = 0\\
					10i-2, &i \in [1,n-1] ^{\dag} \\
					10i-8,  &i \in [n,3n-1]
				\end{cases}$ &
				$10i, i \in [0,3n]$ \\ \hline
		\begin{center}$ \tau( M_{i}) - \tau(m_{i})$\end{center} & $\begin{cases}
			1, & i \in [0,2n-1]\\ 
			2, & i \in [2n, 3n-2]^{\dag}
		\end{cases}$ & $\begin{cases}
			1, & i \in [0,2n-1]\\ 
			2, & i \in [2n, 3n-1] 
		\end{cases}$ \\ \hline
		\begin{center}$\tau( M_{i}) - \tau(m_{i+1})$\end{center} & $ \begin{cases}
			2, & i \in [0,n-2]^{\dag}\\
			1, & i \in [n-1, 3n-2] 
		\end{cases}$ & $\begin{cases}
			2, & i \in [0,n-1]\\
			1, & i \in [n, 3n-1] 
		\end{cases}$ \\ \hline
		\begin{center}$\tau(M_{i})$\end{center} & $\begin{cases}
			1-i, &  i \in [0,n-1]\\
			2-n, & i \in [n, 2n-2]^{\dag}\\
			3 - 3n + i, & i \in [2n-1, 3n-2]
		\end{cases}$& $ \begin{cases}
			1-i, &  i \in [0,n-1]\\
			1-n, & i \in [n, 2n-1]\\
			2 - 3n + i, & i \in [2n, 3n-1]
		\end{cases}$ \\ \hline
		\begin{center}$\tau(m_{i})$\end{center} &  $\begin{cases}
			-i, & i \in [0,n-1]\\
			2 - n, & i \in [n,2n-1]\\
			3 - 3n + i, & i \in [2n, 3n-1]
		\end{cases}$ & $\begin{cases}
			-i, & i \in [0,n-1]\\
			- n, & i \in [n,2n]\\
			-3n + i, & i \in [2n+1, 3n]
		\end{cases}$ \\ \hline
	\end{tabular}
	\caption{Features of the tau functions for the manifolds $\Sigma(2,5,10n\pm3)$.  Cases marked with ``$\dag$'' only appear when $n > 1$.}
	\label{tab:tau5}
\end{table}

\begin{table}
\noindent\makebox[\textwidth]{
	\begin{tabular}[1.5\textwidth]{|m{1.05in}|c|c|}
	\hline
		\begin{center}manifold\end{center} & $\Sigma(2,7,14n-3)$ & $\Sigma(2,7,14n+3)$\\ \hline
		\begin{center}$M_{i}$\end{center} &
			$\begin{cases}
				14i+1, & i \in [0,2n-1]\\
				14(n + \frac{i}{2}) - 5 & i \in [2n,4n-2], i \text{ even}\\
				14(n + \frac{i-1}{2}) + 1 & i \in [2n,4n-2], i \text{ odd}\\
				14(i-n) + 9,  &i \in [4n-1,6n-3]
			\end{cases}$ &
			$\begin{cases}
				14i+1, & i \in [0,2n]\\
				14(n + \frac{i}{2}) + 1 & i \in [2n+1,4n], i \text{ even}\\
				14(n + \frac{i-1}{2}) + 7 & i \in [2n+1,4n], i \text{ odd}\\
				14(i-n) + 1,  &i \in [4n+1,6n]
			\end{cases}$ \\ \hline
			\begin{center}$m_{i}$\end{center}&
			$\begin{cases}
				14i, & i \in [0,2n-1]\\
				14(n + \frac{i}{2}) - 8 & i \in [2n,4n-1], i \text{ even}\\
				14(n + \frac{i-1}{2}) & i \in [2n,4n-1], i \text{ odd}\\
				14(i-n),  &i \in [4n,6n-2]
			\end{cases}$ &
			$\begin{cases}
				14i, & i \in [0,2n]\\
				14(n + \frac{i}{2}) & i \in [2n+1,4n+1], i \text{ even}\\
				14(n + \frac{i-1}{2})+6 & i \in [2n+1,4n+1], i \text{ odd}\\
				14(i-n)-8,  &i \in [4n+2,6n+1]
			\end{cases}$ \\ \hline
		\begin{center}$ \tau( M_{i}) - \tau(m_{i})$\end{center} & $\begin{cases}
			1, & i \in [0,4n-2]\\ 
			2, & i \in [4n-1, 5n-2]\\
			3, & i \in [5n-1,6n-3]^{\dag}
		\end{cases}$ & $\begin{cases}
			1, & i \in [0,4n]\\ 
			2, & i \in [4n+1, 5n]\\
			3, & i \in [5n+1,6n]
		\end{cases}$ \\ \hline
		\begin{center}$\tau( M_{i}) - \tau(m_{i+1})$\end{center} & $ \begin{cases}
			3, & i \in [0,n-2]^{\dag}\\ 
			2, & i \in [n-1, 2n-2]\\
			1, & i \in [2n-1,6n-3]
		\end{cases}$ & $\begin{cases}
			3, & i \in [0,n-1]\\ 
			2, & i \in [n, 2n-1]\\
			1, & i \in [2n,6n]
		\end{cases}$ \\ \hline
		\begin{center}$\tau(M_{i})$\end{center} & $\begin{cases}
			1-2i, &  i \in [0,n-1]\\
			2-n-i, & i \in [n, 2n-1]\\
			3 - 3n, & i \in [2n, 4n-3]^{\dag}\\
			5-7n+i, & i \in [4n-2, 5n-2]\\
			7-12n+2i, & i \in [5n-1, 6n-3]^{\dag}
		\end{cases}$& $ \begin{cases}
			1-2i, &  i \in [0,n]\\
			1-n-i, & i \in [n+1, 2n-1]\\
			1 - 3n, & i \in [2n, 4n]\\
			1-7n+i, & i \in [4n+1, 5n]\\
			1-12n+2i, & i \in [5n+1, 6n]
		\end{cases}$ \\ \hline
		\begin{center}$\tau(m_{i})$\end{center} &  $\begin{cases}
			-2i, &  i \in [0,n-1]\\
			1-n-i, & i \in [n, 2n-1]\\
			2 - 3n, & i \in [2n, 4n-2]\\
			3-7n+i, & i \in [4n-1, 5n-2]\\
			4-12n+2i, & i \in [5n-1, 6n-2]
		\end{cases}$ & $\begin{cases}
			-2i, &  i \in [0,n]\\
			-n-i, & i \in [n+1, 2n-1]\\
			 - 3n, & i \in [2n, 4n+1]\\
			-1-7n+i, & i \in [4n+2, 5n]\\
			-2-12n+2i, & i \in [5n+1, 6n+1]
		\end{cases}$ \\ \hline
	\end{tabular}}
	\caption{Features of the tau functions for the manifolds $\Sigma(2,7,14n\pm3)$.  Cases marked with ``$\dag$'' only appear when $n > 1$.}
	\label{tab:tau73}
\end{table}

\begin{table}
\noindent\makebox[\textwidth]{
	\begin{tabular}[1.5\textwidth]{|m{1.05in}|c|c|}
	\hline
		\begin{center}manifold\end{center} & $\Sigma(2,7,14n-5)$ & $\Sigma(2,7,14n+5)$\\ \hline
		\begin{center}$M_{i}$\end{center} &
			$\begin{cases}
				14i+1, & i \in [0,2n-1]\\
				14(n + \frac{i}{2}) - 9 & i \in [2n,4n-3], i \text{ even}^{\dag}\\
				14(n + \frac{i-1}{2}) + 1 & i \in [2n,4n-3], i \text{ odd}^{\dag}\\
				14(i-n) + 15,  &i \in [4n-2,6n-4]
			\end{cases}$ &
			$\begin{cases}
				14i+1, & i \in [0,2n]\\
				14(n + \frac{i}{2}) + 1 & i \in [2n+1,4n+1], i \text{ even}\\
				14(n + \frac{i-1}{2}) + 11 & i \in [2n+1,4n+1], i \text{ odd}\\
				14(i-n) + 1,  &i \in [4n+2,6n+1]
			\end{cases}$ \\ \hline
			\begin{center}$m_{i}$\end{center}&
			$\begin{cases}
				14i, & i \in [0,2n-1]\\
				14(n + \frac{i}{2}) - 10 & i \in [2n,4n-2], i \text{ even}\\
				14(n + \frac{i-1}{2}) & i \in [2n,4n-2], i \text{ odd}\\
				14(i-n)+4,  &i \in [4n-1,6n-3]
			\end{cases}$ &
			$\begin{cases}
				14i, & i \in [0,2n]\\
				14(n + \frac{i}{2}) & i \in [2n+1,4n+2], i \text{ even}\\
				14(n + \frac{i-1}{2})+10 & i \in [2n+1,4n+2], i \text{ odd}\\
				14(i-n)-8,  &i \in [4n+3,6n+2]
			\end{cases}$ \\ \hline
		\begin{center}$ \tau( M_{i}) - \tau(m_{i})$\end{center} & $\begin{cases}
			1, & i \in [0,4n-3]\\ 
			2, & i \in [4n-2, 5n-3]\\
			3, & i \in [5n-2,6n-4]^{\dag}
		\end{cases}$ & $\begin{cases}
			1, & i \in [0,4n+1]\\ 
			2, & i \in [4n+2, 5n+1]\\
			3, & i \in [5n+2,6n+1]
		\end{cases}$ \\ \hline
		\begin{center}$\tau( M_{i}) - \tau(m_{i+1})$\end{center} & $ \begin{cases}
			3, & i \in [0,n-2]^{\dag}\\
			2, & i \in [n-1, 2n-2]\\
			1, & i \in [2n-1,6n-4]
		\end{cases}$ & $\begin{cases}
			3, & i \in [0,n-1]\\ 
			2, & i \in [n, 2n-1]\\
			1, & i \in [2n,6n+1]
		\end{cases}$ \\ \hline
		\begin{center}$\tau(M_{i})$\end{center} & $\begin{cases}
			1-2i, &  i \in [0,n-1]\\
			2-n-i, & i \in [n, 2n-1]\\
			3 - 3n, & i \in [2n, 4n-4]^{\dag}\\
			6-7n+i, & i \in [4n-3, 5n-3]\\
			9-12n+2i, & i \in [5n-2, 6n-4]^{\dag}
		\end{cases}$& $ \begin{cases}
			1-2i, &  i \in [0,n]\\
			1-n-i, & i \in [n+1, 2n-1]\\
			1 - 3n, & i \in [2n, 4n+1]\\
			-7n+i, & i \in [4n+2, 5n+1]\\
			-1-12n+2i, & i \in [5n+2, 6n+1]
		\end{cases}$ \\ \hline
		\begin{center}$\tau(m_{i})$\end{center} &  $\begin{cases}
			-2i, &  i \in [0,n-1]\\
			1-n-i, & i \in [n, 2n-1]^{\dag}\\
			2 - 3n, & i \in [2n, 4n-3]^{\dag}\\
			4-7n+i, & i \in [4n-2, 5n-2]\\
			6-12n+2i, & i \in [5n-1, 6n-3]
		\end{cases}$ & $\begin{cases}
			-2i, &  i \in [0,n]\\
			-n-i, & i \in [n+1, 2n-1]^{\dag}\\
			 - 3n, & i \in [2n, 4n+2]\\
			-2-7n+i, & i \in [4n+3, 5n+2]\\
			-4-12n+2i, & i \in [5n+3, 6n+2]
		\end{cases}$ \\ \hline
	\end{tabular}}
	\caption{Features of the tau functions for the manifolds $\Sigma(2,7,14n\pm5)$.  Cases marked with ``$\dag$'' only appear when $n > 1$.}
	\label{tab:tau75}
\end{table}

\newpage
\begin{proof}[Proof of (i)]
\renewcommand{\qedsymbol}{}
We give the proof of the Lemma for $\Sigma(2,7,14n+3)$ and leave the arguments for the other five families as exercises (as they work analogously).

As in the proof of Lemma \ref{lem:tau1}, first define the sequences $M_{i}$ and $m_{i}$ via the expressions given in Table \ref{tab:tau5} (we also define $M_{6n+1}:=70n+15$, although this won't end up being the location of a local extremum).  We'll then  compute $\triangle_{j} = \tau(j+1) - \tau(j)$ in each of four cases:

\textbf{Case 1:} Let $M_{i} \leq j < M_{i+1}$, where $0 \leq i \leq 2n-1$.  Now we can write $j = 14 i + 1 + k$, where $0 \leq k \leq 13$.  Equation \ref{eqn:SFtau1} indicates that
\begin{align}\label{eqn:7tau}
	\triangle_{j} &= 1 + j - \left( \left\lceil \frac{j}{2} \right\rceil + \left\lceil \frac{j}{7} \right\rceil + \left\lceil \frac{j(5n+1)}{14n+3} \right\rceil \right)\notag\\
	&=k + 2 - \left( \left\lceil \frac{k+1}{2} \right\rceil + \left\lceil \frac{k+1}{7} \right\rceil + \left\lceil \frac{(k+1)(5n+1) - i}{14n+3} \right\rceil \right)
\end{align}
Clearly $\triangle_{0} = 1$.  Notice that
\begin{gather*}
\begin{aligned}
\left( \left\lceil \frac{k+1}{2} \right\rceil \right)_{k=0}^{13} &= \left( 1,1,2,2,3,3,4,4,5,5,6,6,7,7 \right),\\
\left( \left\lceil \frac{k+1}{7} \right\rceil \right)_{k=0}^{13} &= \left(1,1,1,1,1,1,1,2,2,2,2,2,2,2 \right).
\end{aligned}
\end{gather*}
Now in this case
$$(k+1)(5n+1) + 1 - 2n \leq (k+1)(5n+1)-i \leq (k+1)(5n+1),$$
and one can directly show that
\begin{equation*}
	\left( \left\lceil \frac{(k+1)(5n+1) - i}{14n+3} \right\rceil \right) _{k = 0}^{13}= \left(1,1,s,2,2,3,3,3,4,4,4,5,5,5  \right),
	\text{ where }
	s:=\begin{cases}
		2, & i \in [0,n-1]\\
		1, & i \in [n, 2n-1]
	\end{cases}
\end{equation*}
In light of this, Equation \ref{eqn:7tau} gives that
\begin{equation}\label{eqn:7tau2a}
	\Big( \triangle_{j} \Big)_{j = M_{i}}^{M_{i + 1}-1} = \left( -1, 0,t, 0, 0, 0, 0, 0, -1, 0,0,0,0,1 \right)
	\text{ where }
	t:=\begin{cases}
		-1, & i \in [0,n-1]\\
		0, & i \in [n, 2n-1]
	\end{cases}
\end{equation}
\textbf{Case 2:} Let $M_{i} \leq j < M_{i+1}$, where $4n \leq i \leq 6n+1$.  Now we write $j = 14 (i-n)+ 1 + k$, where $0 \leq k \leq 13$ and
$$(k+1)(5n+1) -5n \leq (k+1)(5n+1)- (i-n) \leq (k+1)(5n+1) - 3n.$$
Via an analysis similar to that in the previous case, one obtains that
\begin{equation}\label{eqn:7tau2b}
	\Big(\triangle_{j} \Big)_{j = M_{i}}^{M_{i + 1}-1} = \left( -1, 0,0,0,0,1,0,0,0,0,0,t, 0, 1 \right)
	\text{ where }
	t:=\begin{cases}
		0, & i \in [4n,5n-1]\\
		1, & i \in [5n, 6n]
	\end{cases}
\end{equation}
\textbf{Case 3:} Let $M_{i} \leq j < M_{i+1}$, where $2n+1 \leq i \leq 4n-1$ and $i$ is odd.  Now $M_{i+1}-M_{i} = 8$, so we write $j = 14(n + \frac{i-1}{2}) + 7 + k$ with $0 \leq k \leq 7$.  One then finds that
\begin{equation}\label{eqn:7tau2c}
	\Big( \triangle_{j} \Big)_{j = M_{i}}^{M_{i + 1}-1} = \left( -1, 0,t, 0, 0, 0, 0, 0, -1, 0,0,0,0,1 \right)
	\text{ where }
	t:=\begin{cases}
		-1, & i \in [0,n-1]\\
		0, & i \in [n, 2n-1]
	\end{cases}
\end{equation}
\textbf{Case 4:} Let $M_{i} \leq j < M_{i+1}$, where $2n+1 \leq i \leq 4n-1$ and $i$ is even.  Now $M_{i+1}-M_{i} = 6$, so we write $j = 14(n + \frac{i}{2}) + 1 + k$ with $0 \leq k \leq 5$.  One then finds that
\begin{equation}\label{eqn:7tau2d}
	\Big( \triangle_{j} \Big)_{j = M_{i}}^{M_{i + 1}-1} = \left( -1,0,0,0,0,1 \right)
\end{equation}

Equations \ref{eqn:7tau2a}-\ref{eqn:7tau2d} imply that $\tau$ is indeed (non-strictly) decreasing on $[M_i,m_{i+1}]$ and (non-strictly) increasing on $[m_{i}, M_{i}]$ for all $i \in [0,6n]$.  It's not hard to show that if $j\geq m_{6n+1}$ that $\triangle_{j} \geq 0$.
\end{proof}
\begin{proof}[Proof of (ii)]
Keeping in mind the expressions for the $M_{i}$ and $m_{i}$, one can use Equations \ref{eqn:7tau2a}-\ref{eqn:7tau2d} to obtain $\tau( M_{i}) - \tau(m_{i})$ and $\tau( M_{i}) - \tau(m_{i+1})$; these in turn give the expressions for $\tau( M_{i})$ and $\tau(m_{i})$.
\end{proof}

\begin{figure}[h!]
\centering
\subfloat[$\Sigma(2,q,2qn+1)$]{
\label{fig:p}
\labellist
\small
\pinlabel* {$\left(\frac{q-1}{2} - 1\right)$ vertices} at 40 60
\pinlabel* {$(n-1)$ vertices} at 170 60
\pinlabel* {$-3$} at 78 20
\pinlabel* {$-1$} at 100 45
\pinlabel* {$-(2q +1)$} at 180 13
\endlabellist
\includegraphics[height = 32mm]{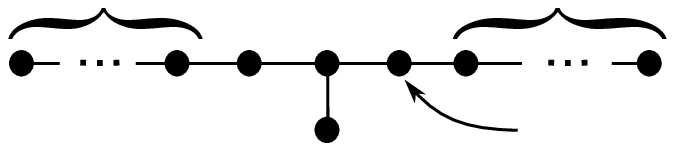}}\\
\subfloat[$\Sigma(2,q,2qn-1)$]{
\label{fig:m}
\labellist
\small
\pinlabel* {$-\left(\frac{q+1}{2}\right)$} at 10 50
\pinlabel* {$(2q-2)$ vertices} at 100 60
\pinlabel* {$(n-2)$ vertices} at 180 60
\pinlabel* {$-3$} at 140 20
\endlabellist
\includegraphics[height = 32mm]{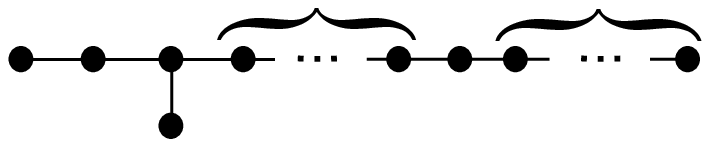}}
\caption{Plumbing graphs for the Brieskorn homology spheres $\Sigma(2,q,2qn\pm1$; unlabelled vertices have weight -2.  For $\Sigma(2,q,2qn-1)$, the graph shown is valid for $n > 1$; when $n =1$, the rightmost arm only has the first $2q-2$ vertices.}
\end{figure}

\begin{figure}[h!]
\centering
\begin{minipage}{0.3\linewidth}
\subfloat[$\Sigma(2,5,10n-3)$]{
\label{fig:5m3}
\labellist
\small
\pinlabel* {$(n-2)$ vertices} at 120 25
\pinlabel* {$-3$} at 82 60
\pinlabel* {$-4$} at 47 60
\pinlabel* {$-1$} at 26 60
\pinlabel* {$-5$} at 37 20
\endlabellist
\includegraphics[width = 50mm]{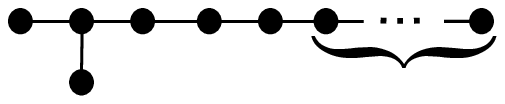}}\\
\subfloat[$\Sigma(2,5,10n+3)$]{
\label{fig:5p3}
\labellist
\small
\pinlabel* {$(n-1)$ vertices} at 100 45
\pinlabel* {$4$ vertices} at 60 30
\pinlabel* {$-5$} at 66 76
\endlabellist
\includegraphics[width = 50mm]{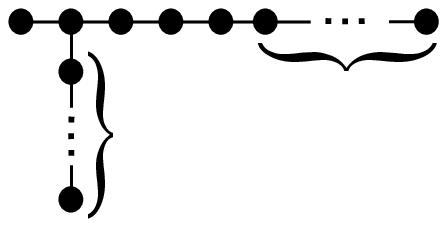}}
\end{minipage}
\begin{minipage}{0.3\linewidth}
\subfloat[$\Sigma(2,7,14n+3)$]{
\label{fig:7p3}
\labellist
\small
\pinlabel* {$(n-1)$ vertices} at 120 55
\pinlabel* {$-1$} at 31 42
\pinlabel* {$-3$} at 52 42
 \pinlabel* {$-6$} at 73 42
 \pinlabel* {$-7$} at 40 7
\endlabellist
\includegraphics[width = 50mm]{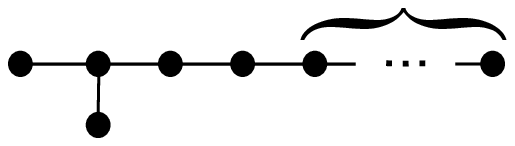}}\\
\subfloat[$\Sigma(2,7,14n-3)$]{
\label{fig:7m3}
\labellist
\small
\pinlabel* {$(n-2)$ vertices} at 125 45
\pinlabel* {$6$ vertices} at 63 35
\pinlabel* {$-3$} at 52 77
\pinlabel* {$-3$} at 90 77
\endlabellist
\includegraphics[width = 50mm]{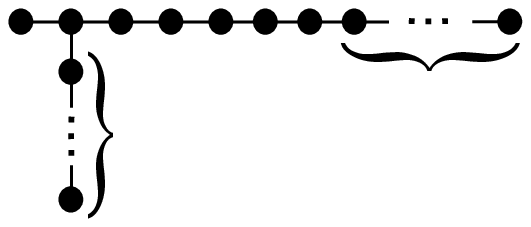}}
\end{minipage}
\begin{minipage}{0.3\linewidth}
\subfloat[$\Sigma(2,7,14n+5)$]{
\label{fig:7p5}
\labellist
\small
\pinlabel* {$(n-1)$ vertices} at 120 25
\pinlabel* {$-1$} at 32 60
\pinlabel* {$-5$} at 53 60
\pinlabel* {$-4$} at 74 60
\pinlabel* {$-4$} at 22 33
\endlabellist
\includegraphics[width = 50mm]{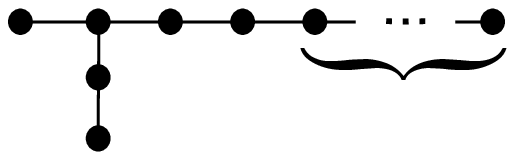}}\\
\subfloat[$\Sigma(2,7,14n-5)$]{
\label{fig:7m5}
\labellist
\small
\pinlabel* {$(n-1)$ vertices} at 130 30
\pinlabel* {$-3$} at 77 65
\pinlabel* {$-3$} at 92 65
\pinlabel* {$-3$} at 30 8
\endlabellist
\includegraphics[width = 50mm]{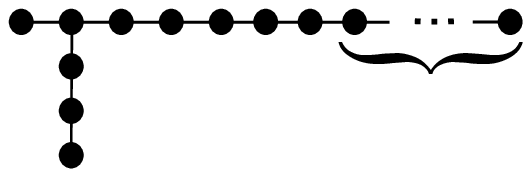}}
\end{minipage}
\caption{Plumbing graphs for the Brieskorn homology spheres $\Sigma(2,5,k)$ and $\Sigma(2,7,k)$.  For $\Sigma(2,5,10n-3$, $\Sigma(2,7,14n-3)$, and $\Sigma(2,7,14n-5)$, the graph shown is valid for $n > 1$ only; when $n = 1$, both the ``$(n-2)$-tail'' and the next vertex inward are missing from the rightmost arm.}
\end{figure}
\newpage
\bibliography{references}
\end{document}